\documentclass[]{article}

\usepackage{epstopdf}
\usepackage[caption=false]{subfig}

\usepackage[numbers,sort&compress]{natbib}
\bibpunct[, ]{[}{]}{,}{n}{,}{,}
\renewcommand\bibfont{\fontsize{10}{12}\selectfont}
\makeatletter
\def\NAT@def@citea{\def\@citea{\NAT@separator}}
\makeatother

\usepackage[a4paper]{geometry}
\usepackage{amssymb}
\usepackage{amsbsy}
\usepackage{amsmath}
\usepackage{amsthm}
\usepackage{graphicx}
\usepackage{epsfig}
\usepackage{amsmath, amssymb, dsfont}
\usepackage{textcomp}
\usepackage{helvet}
\usepackage{stmaryrd}
\SetSymbolFont{stmry}{bold}{U}{stmry}{m}{n}
\usepackage{multirow}
\usepackage{wasysym}
\SetSymbolFont{wasy}{bold}{U}{wasy}{m}{n}
\usepackage{array}
\usepackage{makeidx}
\usepackage{longtable}
\usepackage{units}
\usepackage{rotating}
\usepackage{layouts}
\usepackage{color}
\usepackage{pstricks} 
\usepackage{verbatim}
\usepackage{algorithm}
\usepackage{algorithmicx}
\usepackage{algpseudocode}
\usepackage{mathtools}
\usepackage{enumitem}
\setlist{noitemsep,topsep=0pt} 

\usepackage{hyperref} 

\algrenewcommand{\algorithmicrequire}{\textbf{Input:}}
\algrenewcommand{\algorithmicensure}{\textbf{Output:}}
\algrenewcommand{\algorithmicif}{\textbf{If}}
\algrenewcommand{\algorithmicfor}{\textbf{For}}
\algrenewcommand{\algorithmicwhile}{\textbf{While}}
\algrenewcommand{\algorithmicend}{\textbf{End}}
\algrenewcommand{\algorithmicfunction}{\textbf{Function}}
\algrenewtext{EndIf}{\algorithmicend\ \MakeLowercase{\algorithmicif}}
\algrenewtext{EndFunction}{\algorithmicend\ \MakeLowercase{\algorithmicfunction}}
\algrenewtext{EndFor}{\algorithmicend\ \MakeLowercase{\algorithmicfor}}
\algrenewtext{EndWhile}{\algorithmicend\ \MakeLowercase{\algorithmicwhile}}

\DeclarePairedDelimiter\abs{\lvert}{\rvert}%
\newtheorem{theorem}{Theorem}[section]
\newtheorem{lemma}[theorem]{Lemma}

\newtheorem{definition}[theorem]{Definition}
\newtheorem{example}[theorem]{Example}

\newtheorem{remark}{Remark}

\newtheorem{assumption}{Assumption}

\newcommand{\T}{^{\textrm{T}}}

\newcommand{\RR}{\mathbb{R}}
\newcommand{\QQ}{\mathbb{Q}}

\newcommand{\arr}{{\ensuremath{\A(\H)}}}

\usepackage{anyfontsize}

\newcommand{\SSn}{\mathcal{S}_n}
\newcommand{\xC}{\mathcal{C}}

\newcommand{\A}{\mathcal{A}}
\renewcommand{\H}{\mathcal{H}}
\newcommand{\xint}{\text{int}}
\newcommand{\bitsize}{\textsl{bitsize}}

\newcommand{\myparagraph}[1]{\paragraph*{#1.}} 

\begin{document}


\title{\bf A class of optimization problems motivated
by \\ rank estimators in robust regression}

\author{{\sc Michal~\v{C}ern\'y}\textsuperscript{a},
{\sc Miroslav~Rada}\textsuperscript{a,b}, 
{\sc Jarom\'\i r Antoch}\textsuperscript{a,c} \\ 
{\sc and Milan~Hlad\'\i k}\textsuperscript{a,d,}\thanks{E-mail: cernym@vse.cz (Michal \v{C}ern\'y), miroslav.rada@vse.cz (Miroslav~Rada), antoch@karlin.mff.cuni.cz (Jarom\'\i r~Antoch) and hladik@kam.mff.cuni.cz (Milan~Hlad\'\i k). Submitted to arXiv on October 13, 2019.}}

\date{\textsuperscript{a}Department of Econometrics, University of Economics,  
\\ Winston Churchill Square 4, 13067 Prague~3, Czech Republic \vspace{.3cm} 
\\ \textsuperscript{b}Department of Financial Accounting and Auditing, 
\\ University of Economics, Winston Churchill Square 4, \\ 13067 Prague~3,  Czech Republic \vspace{.3cm}
\\  
\textsuperscript{c}Department of Probability Theory and Mathematical Statistics, 
\\  Faculty of Mathematics and Physics, Charles University, 
\\   Sokolovsk\'a 83, 18675 Prague~8, Czech Republic \vspace{.3cm}
\\  
\textsuperscript{d}Department of Applied Mathematics, Faculty of Mathematics  and Physics, \\  Charles University, Malostransk\'{e} n\'{a}m\v{e}st\'{\i}~25, \\ 11000 Prague~1, Czech Republic}

\maketitle

\begin{abstract}
A rank estimator in robust regression is a minimizer of a function which depends (in addition to other factors) on the ordering of residuals but not on their values. Here we focus on the optimization aspects of rank estimators. We distinguish two classes of functions: the class with a continuous and convex objective function (CCC), which covers the class of rank estimators known from statistics, and also another class (GEN), which is far more general. We propose efficient algorithms for both classes. For GEN we propose an enumerative algorithm that works in polynomial time as long as the number of regressors is $O(1)$. The proposed algorithm utilizes the special structure of arrangements of hyperplanes that occur in our problem and is superior to other known algorithms in this area. For the continuous and convex case, we propose an unconditionally polynomial algorithm finding the exact minimizer, unlike the heuristic or approximate methods implemented in statistical packages.
\vspace{.2cm}


\end{abstract}

\section{Introduction and statistical motivation}

\subsection{Linear regression}

Consider the linear regression relationship
\begin{equation}
y = X\beta^* + \varepsilon,
\label{eq:regr}
\end{equation}
where $y \in \mathbb{R}^n$ stands for the vector of observations of the dependent variable, $X \in \mathbb{R}^{n\times p}$ is the matrix of regressors, $\beta^* \in \mathbb{R}^p$ is the vector of regression parameters and $\varepsilon \in \mathbb{R}^n$ is the vector of (random) errors. We assume that we are given data $(X,y)$ and the task is to 
find an estimate $\widehat{\beta} \equiv \widehat{\beta}(X,y)$ of the unknown vector $\beta^*$ of the true regression coefficients.

There is a rich statistical theory on the estimators $\widehat\beta$
with many desirable properties, such as unbiasedness or consistency, based on various stochastic assumptions on the random errors 
$\varepsilon$. 
In this paper we do not need to go into details: we treat an important class of estimators studied in robust statistics, 
called \emph{rank estimators}, as a particular class of optimization problems.
However, 
the statistical theory is essential for understanding \emph{under which circumstances the usage of rank estimators is appropriate}. 
We recommend \cite{hettmansperger:2010:RobustNonparametricStatistical} 
as a reference book with an exhaustive list of sources, and also the seminal 1970's papers \cite{Jureckova, Jaeckel, technometrics}.

This text talks about rank estimators only from the optimization viewpoint: the question addressed here is \emph{how to compute them efficiently}. 

From the optimization viewpoint, we formulate the problem at a 
higher level of generality
than it is usual in statistics. We will define two classes of functions
\begin{itemize}
\item $\textsf{GEN}$ (``GENeral case'') and
\item $\textsf{CCC}$ (``Continuous and Convex Case'') 
\end{itemize}
and design
algorithms for them. Rank estimators, motivating the definition, form a subclass of $\textsf{CCC}$. 
A detailed description of $\textsf{GEN}$ and $\textsf{CCC}$ will 
be given in Definitions~\ref{def:GEN} 
and \ref{def:CCC}, respectively.

\subsection{Rank estimators}

A \emph{rank estimator} of regression parameters in 
\eqref{eq:regr} can be defined as a minimizer of the optimization problem
\begin{equation}
\min_{\beta \in \mathbb{R}^n} 
 F(\beta) \coloneqq \sum_{i=1}^n a_i^{\beta} (y_i - x_i\T\beta), 
\label{eq:F}
\end{equation}
where $x_i\T$ is the $i$th row of $X$ and $a_1^{\beta}, \dots, 
a_n^{\beta}$ are real-valued coefficients depending on $\beta$ in a specific way: the coefficients $a_1^{\beta}, \dots, a_n^{\beta}$ 
depend \emph{only 
on the ordering of residuals} 
$$r_i^\beta \coloneqq y_i - x_i\T\beta,
$$ 
\emph{but not on their values}. Let us state it more formally.

\begin{assumption} \label{ass:rank}
Let $\SSn$ denote the set of permutations of $\{1, \dots, n\}$.
We assume that 
\begin{itemize}
\item[(a)] the matrix $(X,y)$ contains no pair of identical rows,
\item[(b)]
the function 
$\beta \mapsto (a_1^{\beta}, \dots, a_n^{\beta})$ fulfills the following property: if $\beta, \beta' \in \mathbb{R}^p$ and 
$\pi \in \SSn$ satisfy
\begin{equation}
r^{\beta}_{\pi(1)} <
r^{\beta}_{\pi(2)} < \cdots <
r^{\beta}_{\pi(n)}
\quad \text{and}\quad
r^{\beta'}_{\pi(1)} <
r^{\beta'}_{\pi(2)} < \cdots <
r^{\beta'}_{\pi(n)},
\label{eq:perm}
\end{equation}
then $a_i^{\beta} = a_i^{\beta'}$ for all $i = 1, \dots, n$.
\end{itemize}
\end{assumption}

\begin{remark} We adopt Assumption \ref{ass:rank}(a) as an auxiliary property to make the forthcoming presentation as simple as possible. However,
the property can be regarded as excessively restrictive because identical observations can occur frequently in various statistical applications. 
It is obvious that if $(x_j\T, y_j) = (x_k\T, y_k)$ for some $j \neq k$, then $r_j^{\beta} = r_k^{\beta}$ for all $\beta$ and the strict inequalities in (\ref{eq:perm}) can never be satisfied. However, we show that 
Assumption~\ref{ass:rank}(a) can be deleted if Assumption~\ref{ass:rank}(b) is appropriately (but less transparently) reformulated. Consider Assumption~1(b$'$)
in the following form: \emph{whenever $\beta, \beta' \in \mathbb{R}^n$ and $\pi \in \SSn$ satisfy 
$r^{\beta}_{\pi(j)} < r^{\beta}_{\pi(k)}$ and $r^{\beta'}_{\pi(j)} < r^{\beta'}_{\pi(k)}$ for all 
$1 \leq j < k \leq n$ such that
$(x_{\pi(j)}\T, y_{\pi(j)}) \neq (x_{\pi(k)}\T, y_{\pi(k)})$, then 
$a_i^{\beta} = a_i^{\beta'}$ for all $i = 1, \dots, n$}. Then it is not difficult to verify that the main results will remain valid.
\end{remark}


First we will study the function $F(\beta)$ at a high level of generality (``\textsf{GEN}''), but later---in Section~\ref{sect:CCC}---we will impose further assumptions on $F(\beta)$ implying continuity and convexity (``\textsf{CCC}'').
Before we do that, to put our work in the context of statistics which motivated this analysis, we shortly comment on the usual choices of the $a$-coefficients in robust estimation of regression models.

\subsection{Example: Typical rank estimators in robust statistics}
\label{sect:statistika}

In the theory of rank estimators \cite{hettmansperger:2010:RobustNonparametricStatistical, Jureckova, Jaeckel, technometrics}, 
it is usual 
to fix in advance a family of coefficients $\alpha_1 \le \alpha_2 \le \dots\le \alpha_n$
such that $\sum_{i=1}^n \alpha_i = 0$ and define
$$
a_i^\beta = \alpha_{\pi^{-1}(i)},
$$
where $\pi \in \SSn$ is any permutation satisfying 
$r^{\beta}_{\pi(1)} \leq
r^{\beta}_{\pi(2)} \leq \cdots \leq
r^{\beta}_{\pi(n)}$. It can be shown that if there are more such permutations $\pi$, a choice of any of them results in the same value 
$F(\beta) = \sum_{i=1}^n \alpha_{\pi^{-1}(i)}r_i^{\beta}$ (for details see Section~\ref{sect:CCC}). It follows that $F(\beta)$ is a well-defined function. In addition, in 
Section~\ref{sect:CCC}
we will show that the function $F(\beta)$ is piecewise linear, continuous and convex; and it is also easy to show that it is bounded from below. It means that the minimizer of $F(\beta)$ --- the rank estimator $\widehat\beta$ ---
exists (although it need not be unique).

The typical choices are of the form $\alpha_i = \varphi(i/(n+1))$, $i = 1, \dots, n$, where
$\varphi(t)$ is a \emph{score function} satisfying some regularity properties, such as 
$\int_0^1 \varphi(t)\, \text{d}t = 0$.\footnote{Sometimes it is also required $\int_0^1 \varphi^2(t)\, \text{d}t = 1$
and even central symmetry of the graph of $\varphi$ around $[1/2, 0]$; such properties are essential for asymptotic statistical properties, 
but do not play a role in this text.}
Some well-known representatives include
\begin{itemize}
    \item
$\varphi(t) = \text{sgn}(t-1/2)$ (sign score function), 
\item $\varphi(t) = \sqrt{12}(t-1/2)$ (Wilcoxon score function) or
\item $\varphi(t) = \Phi^{-1}(t-1/2)$ (van der Waerden score function), 
where $\Phi^{-1}$ is the quantile function of the standard Gaussian distribution. 
\end{itemize}

\subsection{Timetable}

The structure of the paper is as follows:
\begin{itemize}
\item
\textbf{Section \ref{sect:general}: The general case (\textsf{GEN}) --- a polynomial method as long as $p = O(1)$.} (Recall that $p$ stands for the number of regressors.)
Here, we will consider the problem 
\eqref{eq:F} at a high level of generality. That is,
we will not impose further assumptions on the $a$-coefficients in
\eqref{eq:F}. The function $F$ can be very general and tough; it need not be continuous and there can be a different set of $a$-coefficients corresponding to each permutation $\pi$ of residuals, see \eqref{eq:perm}. Then, minimization of $F$ is a combinatorial optimization problem. But still, the main result of Section~\ref{sect:general} 
is a design of a much better algorithm than the 
straightforward brute-force enumeration of permutations $\pi \in \SSn$. 
As a corollary, we get a polynomial-time method as long as $p = O(1)$. 

\item
\textbf{Section \ref{sect:CCC}: The continuous and convex case (\textsf{CCC}) --- an unconditionally polynomial method.}
We impose additional assumptions on $F$ implying continuity and convexity. In this case we design an unconditionally polynomial method. The method works for a broader class of functions that those described in Section~\ref{sect:statistika}; for example, the class involves also $F$-functions unbounded from below. (Thus, we will have to design a method to test unboundedness, even if $F$-functions in statistics are always bounded from below.)
\end{itemize}

\section{Geometry of arrangements of hyperplanes}\label{sect:arrng}

\subsection{Cells of an arrangement}

Given a set $A \subseteq \mathbb{R}^p$, let $\xint(A)$ denote its interior.
Recall that $r_i^{\beta}$ is a shorthand for the $i$th residual of the form $r_i^{\beta} = y_i - x_i\T\beta$.
Given a permutation $\pi \in \SSn$, define
\begin{equation}
    C^{\pi} \coloneqq \left\{\beta \in \mathbb{R}^p\ \Big| \ 
r^{\beta}_{\pi(1)} \leq 
r^{\beta}_{\pi(2)} \leq \cdots \leq
r^{\beta}_{\pi(n)}\right\}.
\label{eq:Cpi}
\end{equation}
Observe that $C^{\pi}$ is a convex polyhedron in $\mathbb{R}^p$,
that $\bigcup_{\pi\in\SSn}C^{\pi} = \mathbb{R}^p$ and 
that $\xint(C^{\pi}) \cap \xint(C^{\pi'}) = \emptyset$ whenever $\pi \neq \pi'$. 
If the polyhedron $C^{\pi}$ is full-dimensional, we call it 
\emph{cell}. 

The term comes from the theory of arrangements of hyperplanes. Consider the system $\H$ of hyperplanes
\begin{equation}
\label{eq:hyperplanes}
\underbrace{
 H_{ij} \coloneqq    \left\{\beta\in\RR^p\  \Big|\ r_i^\beta = r_j^\beta\right\} 
 = \left\{\beta\in\RR^p\ \Big|\ (x_i - x_j)\T\beta = y_i - y_j \right\}, \ 1 \leq i < j \leq n}_{\eqqcolon\H}.
\end{equation}
The system of hyperplanes $\H$ ``cuts'' the space $\mathbb{R}^p$ into a finite number of full-dimensional regions
and these are usually called (full-dimensional) cells. (We will drop the adjective ``full-dimensional'' for brevity.) More precisely speaking, 
$$
\RR^n \setminus \bigcup_{1 \leq i < j \leq n} H_{ij} = \bigcup_{\pi \in \SSn} \xint(C^{\pi}).
$$
Let $\xC$ denote the system of cells (i.e., the set of permutations $\pi \in \SSn$ such that $C^{\pi}$ is full-dimensional). 
Lemma~\ref{lem:number:cells} gives a bound on the size of $\xC$.

\begin{lemma}[\cite{buck:1943:Partitionspace}, \cite{zaslavsky:1975:FacingarrangementsFacecount}]
    \label{lem:number:cells}
An arrangement of $N$ hyperplanes in $\mathbb{R}^p$
has at most $\zeta(N,p) \coloneqq \sum_{i=0}^{p} \binom{N}{i}$ 
(full-dimensional) cells.
\end{lemma}
We have $|\xC| \leq \zeta(N,p)$ and this number 
can be easily bounded by $O(N^{p})$.
The arrangement \eqref{eq:hyperplanes} has $N = \binom{n}{2} = \frac{1}{2}n(n-1)$ hyperplanes;
thus 
\begin{equation}\label{eq:zeta}
|\xC| \leq \zeta\left(\tfrac{1}{2}n(n-1),p\right) = O(n^{2p}). 
\end{equation}

\myparagraph{Notation}
If $\mathcal{S}$ is a system of hyperplanes, then $(\mathcal{S})$ is a shorthand for the arrangement defined by 
the hyperplanes in $\mathcal{S}$. In particular, the arrangement defined by \eqref{eq:hyperplanes} is denoted by \arr{}.


\subsection{Faces of a general dimension}

If we are given a nonempty system of cells $\pi^1, \pi^2, \dots, \pi^k \in \xC$
and $\Phi \coloneqq C^{\pi^1} \cap C^{\pi^2} \cap \cdots
\cap C^{\pi^k} \neq \emptyset$, then the set $\Phi$ is called \emph{face} 
of the arrangement~\arr. The \emph{dimension} of a face is defined as its affine dimension. 

For example, a cell is a $p$-dimensional face. A cell of dimension $p-1$ is called \emph{facet}. 
If the face $C^{\pi} \cap C^{\pi'}$ is a facet, we say that cells
$C^{\pi}$ and $C^{\pi'}$ are \emph{neighbors}. A face of zero dimension is called \emph{vertex}.

\section{The general case (\textsf{GEN})}\label{sect:general}

\subsection{A regularity condition}

Assumption~\ref{ass:rank} implies that the $a$-coefficients
are constant on $\xint(C^\pi)$ for a $\pi \in \xC$. That is,
for $\pi\in\xC$,
the restriction of $F(\beta)$ onto $\xint(C^\pi)$ is a linear function.
Thus, when restricted to $\bigcup_{\pi\in\xC}\xint(C^{\pi})$, the function $F(\beta)$ is piecewise linear.
However, Assumption~\ref{ass:rank} does not restrict how the function $F$ should behave on faces of dimension $\leq p -1$. We need to control the behavior of $F(\beta)$ on them. Put another way, Assumption \ref{ass:rank} determines the behaviour
of $F(\beta)$ on $\bigcup_{\pi\in\xC}\xint(C^{\pi})$ and we need to
impose a ``reasonable'' condition on how to extend $F(\beta)$ 
also to
$\mathbb{R}^p \setminus \bigcup_{\pi\in\xC}\xint(C^{\pi})$.

\begin{assumption}
$F$ is lower semicontinuous.
\label{as:reg}
\end{assumption}


Said informally: consider that the point $\beta$ 
lies in the intersection of two neighbouring cells (i.e.,~$\beta$ is in a facet) and that
$F$ is discontinuous there. Say that we want $F(\beta)$ to inherit 
the behaviour either from the first or the second cell. We must decide
whether $F(\beta)$ should be the value as if $\beta$ belonged to the 
first or to the second cell. By discontinuity, either of the extensions would result in a different value $F(\beta)$. 
Assumption \ref{as:reg} tells us that the extension resulting in the lower value $F(\beta)$ is preferred. 

    %

\begin{definition} \textsf{GEN} is the class of $F$-functions from
\eqref{eq:F} satisfying Assumptions~\ref{ass:rank} and \ref{as:reg}.
\label{def:GEN}
\end{definition}

\begin{example} Figure~1 depicts 
an example of a function $F \in \textsf{GEN}$ with $\beta \in \mathbb{R}^2$ and $n = 5$. The thick lines
are the hyperplanes 
\eqref{eq:hyperplanes}. There is an arrow placed in an unbounded cell $C^{\pi}$ on which $F(\beta)$ is unbounded from below. (There are more such cells.) 
\end{example}

\subsection{High-level outline of the algorithm}
Assumption \ref{as:reg} implies that minimization of $F(\beta)$
is equivalent to solving the optimization problem
\begin{equation}
\min_{\pi\in\xC} \min_{\beta \in C^{\pi}} 
\sum_{i=1}^n a_i^{\widetilde\beta_\pi}(y_i - x_i\T\beta),
\label{eq:hlavni}
\end{equation}
where $\widetilde\beta_\pi$ is an arbitrary point in $\xint(C^{\pi})$.
By Assumption~\ref{ass:rank}, the vector $(a_1^{\beta}, \dots, a_n^{\beta})$ is the same 
for all $\beta \in \xint(C^{\pi})$ and
the choice of an \emph{arbitrary} representative
$\widetilde\beta_\pi \in \xint(C^{\pi})$ is correct.

It follows that, for a given $\pi \in \xC$, the inner optimization 
problem
\begin{equation}
\min_{\beta \in C^{\pi}} 
a^{\widetilde\beta_\pi}(y_i - x_i\T\beta)\label{eq:inner}
\end{equation}
is a linear programming problem which can be solved efficiently 
(indeed, the constraints ``$\beta \in C^{\pi}$'' are linear).

\begin{figure}[t]
\centering
\includegraphics[width=10cm]{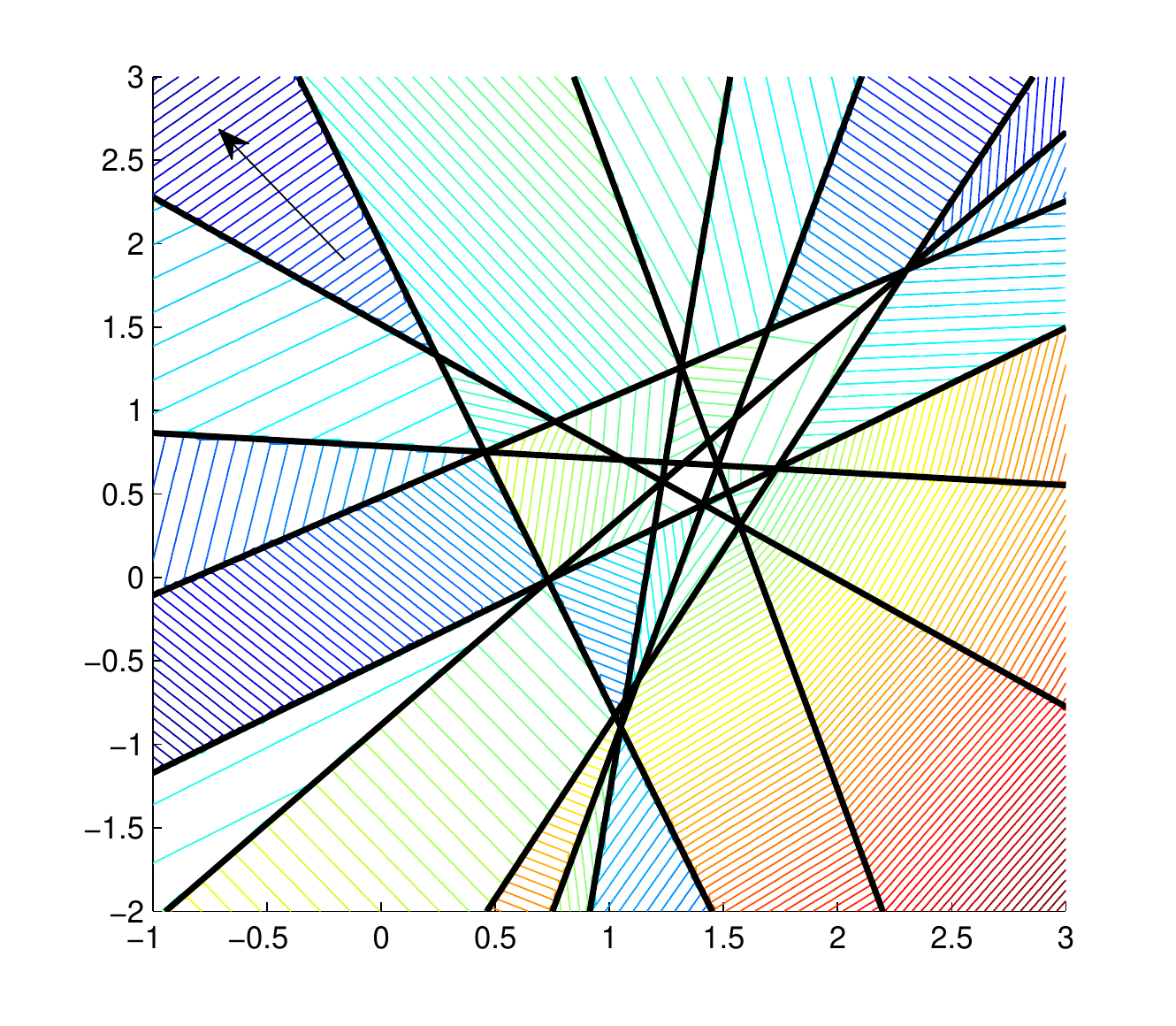}
\caption{An example of the arrangement \arr{} 
in $\mathbb{R}^2$ with $n = 5$
and contour lines of a sample function $F\in\textsf{GEN}$. The arrow depicts an unbounded cell on which $F(\beta)$ is unbounded.}
\end{figure}

The problem \eqref{eq:hlavni} of minimizing $F(\beta)$ reduces to an enumeration
of $\pi \in \xC$. What remains is to show 
that the cells $\pi \in \xC$ can be enumerated in a reasonable way
so that \emph{not all potential permutations $\pi \in \SSn$} are to be tested. Observe that if $p \ll n$, there are many permutations $\pi \in \SSn$
such that $\pi \not\in \xC$; this follows from the bound
(\ref{eq:zeta}). Such permutations need to be avoided and {\em the core of the forthcoming method lies exactly in this}.

In this Section we will assume that
\begin{equation}
p = O(1).\label{eq:poone}
\end{equation}
Then the bound $O(n^{2p})$ from (\ref{eq:zeta}) is a polynomial in $n$. We will describe a method
for enumeration of $\xC$ in time just slightly worse than $O(n^{2p})$; see Theorem~\ref{the:correctness:and:complexity}\ref{enu:the:enumeration:complexity} for details. We get the main result of Section~\ref{sect:general}:

\begin{theorem}
If $p = O(1)$, a function $F \in \textsf{GEN}$ can be minimized in polynomial time.
\end{theorem}

In Section \ref{sec:technical:issue} we discuss a problem of representation of input. Section \ref{sec:enumeration:cells:general} describes the enumeration algorithm itself. 

\subsection{A technical issue: Representation of the input}
\label{sec:technical:issue}

Under Assumptions \ref{ass:rank} and \ref{as:reg}, every fulldimensional cell $C^\pi$ can be equipped with its own set of $a_i^{\beta}$-coefficients. Since every cell is determined by a permutation $\pi \in S_n$, the number of the sets of $a_i^{\beta}$ coefficients can be as huge as $n!$. If these coefficients are considered as a part of an input instance, ``everything'' becomes polynomially solvable.


We thus assume that, given $\pi\in\SSn$, the $a$-coefficients are returned from an external oracle or computed on-the-fly (but the time for their computation is disregarded). We say that the bitsize of the input corresponds to the bitsize $L$ of data $(X,y)$ (which are assumed to be rational) and that the bitsize of the vector 
$(a^{\beta}_1, \dots, a^{\beta}_n)$ is also bounded by a polynomial in $L$ for every $\beta$.

Here we use the term ``bitsize'' informally. A reader can check Section~\ref{sect:BigL}, where bitsizes of rational numbers and some Big-$L$ arguments are elaborated on with more details.

\subsection{Enumeration of cells}
\label{sec:enumeration:cells:general}

We propose a new algorithm called \emph{RSIncEnu} for enumeration of cells of the arrangement \arr. 
RSIncEnu modifies the existing algorithm FlIncEnu \cite{rada:2018:NewAlgorithmEnumeration} 
designed for general arrangements. The modification utilizes special properties of \eqref{eq:hyperplanes}. 
The key property is that every cell can have only a limited number of neighbor cells compared to general arrangements, see Lemma \ref{lem:numnei} for details. We prove that RSIncEnu has a better time complexity than FlIncEnu and that it is superior to other known algorithms with similar properties.

More precisely, FlIncEnu (and, in turn, RSIncEnu) has two important properties: \emph{output sensitivity} and \emph{compactness}. Output sensitivity means that time complexity can be related not only to the size of input, but also to the size of the output.\footnote{On the contrary, algorithms for decision problems give just a single-bit YES/NO output. This is an example where the computation time cannot be related to the size of output.}  This allows us 
to express time complexity per unit of output (in our case: per cell in~$\xC$). This is particularly useful for enumeration algorithms, where the size of output is often exponential w.r.t.~the size of input. 

Compactness means that space complexity is polynomial in the size of input. In particular, during the computation, it is not necessary to store the entire output constructed so far. (In other words, the output can be printed as a stream.) This is fully satisfactory for our problem---for a cell we just solve the LP \eqref{eq:inner} and then the cell can be dropped.

Aside of FlIncEnu, there exist other output sensitive and compact algorithms for enumeration of cells of arrangements, namely ReverseSearch algorithm and its improvements \cite{avis:1996:Reversesearchenumeration,sleumer:1998:Outputsensitivecellenumeration,ferrez:2005:Solvingfixedrank}. Actually, RSIncEnu can be proven to be superior to them; the proof idea is analogous to \cite{rada:2018:NewAlgorithmEnumeration}.

\myparagraph{Number of neighbours}
Recall that the arrangement \arr{} is given by $\binom{n}{2}$ hyperplanes $H_{ij}$, $1 \leq i < j \leq n$.
Each hyperplane $H_{ij}$ determines two halfspaces 
\begin{equation}
    H^-_{ij} \coloneqq \left\{\beta\ \Big|\ r_i^\beta \le r_j^\beta\right\} \text{\ \ and\ \ } H^+_{ij} \coloneqq \left\{\beta\ \Big|\ r_i^\beta \ge r_j^\beta\right\}.
    \label{eq:halfspaces}
\end{equation}

By \eqref{eq:Cpi}, a cell is defined as an intersection of $n-1$ halfspaces only (since the membership to other halfspaces can be decided by transitivity). We get an important property:

\begin{lemma}
\label{lem:numnei}
A cell of \arr{} can have at most $n-1$ neighbor cells.
\end{lemma}

Lemma~\ref{lem:numnei} shows a difference between arrangement \arr{} and a general arrangement $\mathcal{A}(\mathcal{S})$. In a general arrangement, all of the $|\mathcal{S}|$ halfspaces might be necessary to describe a cell. Lemma~\ref{lem:numnei} 
captures the crucial property leading to a ``good'' algorithm.

\myparagraph{Redundant hyperplanes}
Note that if there are two identical hyperplanes in an arrangement, one of them can be removed without affecting the structure of the arrangement. In our case, this can occur if
$x_i - x_j = \eta (x_k - x_{\ell})$ and 
$y_i - y_j = \eta (y_k - y_{\ell})$ 
for some distinct indices $i,j,k,\ell$ 
and $\eta \neq 0$.
 
A hyperplane $H_{k\ell}$ is said to be \emph{redundant w.r.t. $H_{ij}$} if $H_{ij} = H_{k\ell}$ and $(i,j) \prec  (k,\ell)$, where $\prec$ is a fixed ordering, say lexicographic.

To simplify the forthcoming discussion, define the list $\mathcal{R}$ of redundant hyperplanes as 
\begin{align}
    \label{eq:R}   \mathcal{R}&\coloneqq \{ \{k,\ell\}\ |\ \exists(i,j)\ \text{s.t.}\ 1 \leq i < j \leq n,\ H_{ij} = H_{k\ell}\ \text{and}\ (i,j) \prec (k,\ell) \}.
\end{align}

Clearly, $\arr$ and $\A(\H\setminus \mathcal{R})$ are the same arrangements.

\myparagraph{Tight hyperplanes}
Recall that neighbor cells are cells that share a common facet. In other words, if $C^{\pi_1}$ and $C^{\pi_2}$ are neighbor cells, they are separated by \emph{exactly one} hyperplane from $\mathcal{H} \setminus \mathcal{R}$, say $H_{ij}$. The permutations $\pi_1, \pi_2$ differ in the transposition of elements $i$ and $j$ (and also in transpositions of pairs $\{k,\ell\}$ such that $H_{k\ell}$ is a redundant hyperplane w.r.t. $H_{ij}$). 

The following property of neighbor cells will be useful (the proof is easy): 
\begin{lemma}
    \label{def:tight}
    Let $C^{\pi_1}$ and $C^{\pi_2}$ be neighbor cells. Then there are points $\beta^1 \in \xint(C^{\pi_1})$ and $\beta^2 \in \xint(C^\pi_2)$ such that the line segment $\beta_1$--$\beta_2$ crosses exactly one hyperplane from $\mathcal{H}\setminus\mathcal{R}$. 
\end{lemma}

The hyperplane from Lemma~\ref{def:tight} is called \emph{tight} (for $C^{\pi_1}$ and $C^{\pi_2}$).

\begin{lemma}
    \label{lem:tight:hyperplanes}
        If $H$ is a tight hyperplane of a cell $C^\pi$, then
     \begin{equation}
         \label{eq:tight}
         H \in \Big\{H_{ij}\ \Big|\  1\leq i < j \leq n,\ \{i,j\} \in \{ \{\pi(k),\pi(k+1)\}\ |\ k = 1,\ldots, n-1\}\Big\}.
     \end{equation}
\end{lemma}
\begin{proof} 
    For contradiction assume that there is a tight hyperplane $H_{k\ell}$ not satisfying \eqref{eq:tight}. Let $C^{{\pi'}}$ be the neighbor cell of $C^\pi$ with the tight hyperplane $H_{k\ell}$. By Lemma~\ref{def:tight}, there have to be points $\beta \in \xint(C^\pi)$ and $\beta' \in \xint(C^{\pi'})$
such that the line segment $\beta$--$\beta'$ intersects $H_{k\ell}$. Let $\beta''$ be the intersection point.

If $H_{k\ell}$ does not satisfy \eqref{eq:tight}, there has to be $j$ such that 
$$r^\beta_k <r^\beta_j < r^\beta_\ell \text{\ \ \ or\ \ \ } r^\beta_k > r^\beta_j > r^\beta_\ell.
$$ 
Without loss of generality we can assume that $k < j < \ell$. We have $r^{\beta''}_k = r^{\beta''}_\ell$. 
Since $r^{\alpha\beta+(1-\alpha)\beta''}_j$ is a linear function in $\alpha$, either hyperplane $H_{kj}$ or hyperplane $H_{j\ell}$ has to be intersected by the line segment $\beta$--$\beta''$. This contradicts the fact that $H_{k\ell}$ is the only hyperplane crossed by the line segment    $\beta$--$\beta'$.
\end{proof}



\myparagraph{Checking tightness of a hyperplane}
Note that a polyhedon $\{x\ | Ax \le b\}$ contains an interior point if and only if there is an $x^*$ satisfying $Ax^* < b$.

Consider a cell $C^{\pi}$ and a hyperplane $H_{ij}$. The test ``\emph{Is $H_{ij}$ a tight hyperplane of~$C^{\pi}$}?'' can be performed using linear programming:  
\begin{lemma}
    \label{lem:cell:test}
    Let $\pi \in \SSn$ and let $H_{ij}$ be a hyperplane from $\H\setminus R$. The linear program
    \begin{equation}
        \label{eq:tightness:test}
    \begin{aligned}
        \max_{\beta\in\RR^p,\ \varepsilon\in\RR}&&\varepsilon\\
        \text{\rm subject to } && r^\beta_{\pi(k)} + \varepsilon &\le r^\beta_{\pi(k+1)},\qquad k = 1,\ldots, n-1;\ \{\pi(k),\pi(k+1)\} \not \in \mathcal{R} \cup \{ \{i,j \} \},\\
                               && r^\beta_i &= r^\beta_j.
    \end{aligned}
\end{equation}
has a positive optimal value if and only if $H_{ij}$ is a tight hyperplane of $C^{\pi}$.
\end{lemma}

\begin{proof}
    Let $\beta' \in \xint(C^\pi)$ be an interior point of $C^\pi$. Let $(\beta^*,\varepsilon^*)$ be an optimal solution of \eqref{eq:tightness:test}. 
    
    If $\varepsilon^* >0$, we have 
    \begin{align}
        r^{\beta'}_{\pi(k)} &< r^{\beta'}_{\pi(k+1)}&&\qquad k = 1,\ldots, n-1;\\
        r^{\beta^*}_{\pi(\ell)} &< r^{\beta^*}_{\pi(\ell+1)} &&\qquad \ell = 1, \ldots, n-1,\ \{\pi(\ell),\pi(\ell+1)\} \not \in \mathcal{R}\cup\{\{i,j\}\};\\
        r^{\beta^*}_{i} &= r^{\beta^*}_{j}. 
    \end{align}
    Consider the ray $\varrho \coloneqq \{\beta\ |\ \beta = \beta'+\lambda (\beta^*-\beta'),\ \lambda \ge 0\}$ from $\beta'$ towards $\beta^*$. Clearly, the first hyperplane intersected by $\varrho$ is $H_{ij}$. Let $\beta''$ be the next intersection of $\varrho$ and another 
hyperplane. Now set 
    \begin{equation} 
        \beta^0\coloneqq \tfrac{1}{2}(\beta^*+\beta''). 
        \label{eq:beta:new}
    \end{equation}
        By construction, $\beta^0$ is an interior point of some cell, and the line segment $\beta'$--$\beta^0$ crosses exactly one hyperplane, which proves that $H_{ij}$ is a tight hyperplane of $C^{\pi}$.


If $\varepsilon^* < 0$, 
then $C^{\pi} \cap H_{ij} = \emptyset$ and $H_{ij}$ cannot be tight for $C^{\pi}$.

If $\varepsilon^* = 0$, at least one inequality constraint in \eqref{eq:tightness:test} is satisfied as equality.
Thus, every $\beta \in C^\pi \cap H_{ij}$ also satisfies 
$\beta \in H_{\pi(k),\pi(k+1)}$ for some $k$ and the separating hyperplane $H_{ij}$ is not unique.
\end{proof}

Note that the point $\beta^0$ from \eqref{eq:beta:new} is an interior point of a neighbor cell $C^{\pi'}$ of $C^\pi$. 
The permutation $\pi'$ can be obtained by sorting the residuals $r_1^{\beta_0}, \dots, r_n^{\beta_0}$.

\subsection {RSIncEnu algorithm}
RSIncEnu is formalized as Algorithm \ref{alg:rsincenu}. It enumerates cells recursively. A call of RSIncEnu takes two arguments: an interior point $\beta$ of a cell and a list $L$ of hyperplanes identified as tight hyperplanes in previous calls of this branch of recursion. 
Then, hyperplanes \eqref{eq:tight} are tested for tightness using Lemma~\ref{lem:cell:test}. For every tight hyperplane $H$ found:  
\begin{itemize}\item[(i)] we get an interior point $\beta^0$ of a neighbor cell by \eqref{eq:beta:new},
\item[(ii)] add the hyperplane $H$ to the list $L$ of hyperplanes that were identified as tight hyperplanes so far in the current branch of recursion, and
\item[(iii)] we call recursively RSIncEnu$(\beta_0, L)$.
    \end{itemize}

\begin{remark}
As a starting point, $\beta$ can be chosen arbitrarily. If it lies in a hyperplane, it can be easily perturbed to obtain an interior point of a cell.
\end{remark}

\begin{remark}
If we define the graph $G = (V,E)$ with $V = \xC$ and $E = \{ \{C^1,C^2\}\ | \ C^1 \text{ and } C^2$  are 
$\text{neighbor cells}\}$, then the recursive calls of RSIncEnu form a spanning tree of $G$.
\end{remark}

Theorem \ref{the:correctness:and:complexity} proves correctness of Algorithm~\ref{alg:rsincenu}. In particular, we  prove that \emph{each cell is visited exactly once}. 

\begin{algorithm}[ht]
    \begin{algorithmic}[1]
        \Require $X \in \QQ^{n \times p},\ y \in \QQ^n$ and $\beta \in \xint(C^\pi)$ for some $\pi \in \SSn$ 
        \State construct the set $\mathcal{R}$ using \eqref{eq:R}\label{alg:rsincenu:preparation}
        \State \Call{RSIncEnu}{$\beta,\emptyset$}\label{alg:rsincenu:first:call}
        \Statex
        \Function{RSIncEnu}{$\beta; L$}
        \State Compute $\pi$ such that $r^\beta_{\pi(1)} < \cdots < r^\beta_{\pi(n)}$\label{alg:rsincenu:compute:pi}
        \State Output $C^\pi$\label{alg:rsincenu:output}
        \For {$i = 1,\ldots,n-1$}\label{alg:rsincenu:for:start}
        \If {$ \{\pi(i), \pi(i+1)\} \not \in \mathcal{R} \cup L$}\label{alg:rsincenu:hyperplane:check}
        \State $k = \min\{\pi(i),\pi(i+1)\}$;\quad $\ell = \max\{\pi(i),\pi(i+1)\}$\label{alg:rsincenu:swap:indices}
        \State Solve \eqref{eq:tightness:test} for $C^\pi$ and $H_{k\ell}$; let $\beta^*, \varepsilon^*$ be a maximizer \label{alg:rsincenu:tighness:test}
        \If {$\varepsilon^* > 0$} \label{alg:rsincenu:tight:found}
        \State Compute $\beta^0$ according to \eqref{eq:beta:new}\label{alg:rsincenu:new:beta}
        \State $L \coloneqq L \cup \{\{k,\ell\}\}$\label{alg:rsincenu:update:L}
        \State \Call{RSIncEnu}{$\beta^0, L$}\label{alg:rsincenu:recursive:call}
        \EndIf
        \EndIf
        \EndFor\label{alg:rsincenu:for:end}
        \EndFunction
    \end{algorithmic}
    \caption{RSIncEnu}
    \label{alg:rsincenu}
\end{algorithm}

\begin{theorem}
    \label{the:correctness:and:complexity}\hfill
    \begin{enumerate}[label=(\alph*)]
        \item\label{enu:the:enumeration:correctness} RSIncEnu outputs all elements of $\xC$. Every element is output exactly once.
    \item\label{enu:the:enumeration:complexity} The total time complexity of RSIncEnu is $O( n\cdot \abs{\xC}\cdot \mathrm{lp}(n,p) + n^2\log n p)$ and the space complexity is $O(\mathrm{lp}(n,p)+n^2 p)$, where $\mathrm{lp}(n,p)$ is the time or space to solve a linear program with $O(n)$ constraints, $O(p)$ variables and with bitsize of data $O(\textsl{bitsize}(X,y))$.
    \end{enumerate}
\end{theorem}
\begin{proof} \emph{Part \ref{enu:the:enumeration:correctness}.}
Assume we are in the course of an enumeration process. Consider a RSIncEnu call with $\beta'$ is being evaluated and let $L'$ be the current value $L$ (i.e. the current list of hyperplanes that were considered tight so far). Let $C^\pi$ be a cell which has not been enumerated so far and let $\beta \in \xint(C^\pi)$. We say that a cell $C^\pi$ is \emph{reachable} from the current RSIncEnu call if for every $(i,j)$ such that $i<j$ and $\{i,j\} \in L'$ we have $\beta \in H^-_{ij} \Leftrightarrow \beta' \in H^-_{ij}$; recall that $H^-_{ij}$ is one of the halfspaces induced by $H_{ij}$. I.e.,~$\beta$ and $\beta'$ are on the same side of every tight hyperplane found so far.

    Note that during the enumeration process, each cell, say $C^\pi$ with an interior point $\beta$, is always reachable from exactly one RSIncEnu call. At the very first RSIncEnu call, $L$ is empty, so $C^\pi$ is reachable. Then, if a hyperplane, say $H_{ij}$ is added to $L$ on Line \ref{alg:rsincenu:update:L} in some RSIncEnu($\beta', L)$ call, $C^\pi$ is reachable either from the new RSIncEnu call on Line \ref{alg:rsincenu:recursive:call} (if $\beta'$ is on the different side of $H_{ij}$ than $\beta$), or still from the current RSIncEnu call (if $\beta'$ is on the same side of $H_{ij}$).

    It remains to show that every reachable cell is really ``reached'' during the enumeration. Assume that RSIncEnu($\beta', L'$) is being evaluated and let $C^{\pi'}$ be the cell output in this RSIncEnu call. Consider a reachable cell $C^\pi$ with an interior point $\beta$. Consider the line segment $\beta'$--$\beta$. Take the first hyperplane it intersects (the case of tie can be resolved e.g. by perturbation of $\beta$ or by taking one of the firstly intersected hyperplanes that contains a facet of $C^{\pi'}$). This hyperplane is a tight hyperplane of $C^{\pi'}$, which will result to new RSIncEnu call for a $\beta^0$ that is separated from $\beta$ by one hyperplane fewer. Since the total number of hyperplanes is $\binom{n}{2}$, the maximal depth of recursion is also $\binom{n}{2}$.

    \emph{Part \ref{enu:the:enumeration:complexity}: Time complexity.}
    The preprocessing step on Line \ref{alg:rsincenu:preparation} can be performed by normalizing $p$ coefficents of $O(n^2)$ hyperplanes and sorting them. This takes $p\cdot O(n^2) \cdot \log(O(n^2)) = O(p \cdot n^2\cdot \log n)$ time. 
    Then RSIncEnu is called for every cell exactly once. Inside a call, residuals are sorted (Line \ref{alg:rsincenu:compute:pi}; takes time $O(p\cdot n\cdot\log n)$). Then, there are $n-1$ iterations of for-cycle \ref{alg:rsincenu:for:start}--\ref{alg:rsincenu:for:end}. In each iteration, at most one linear programs with at most $n-1$ constraints and $p+1$ variables is solved (Line \ref{alg:rsincenu:tighness:test}). This amounts to the complexity $O(n\cdot\abs{\xC}\cdot \mathrm{lp}(n,p))$. If the test on Line \ref{alg:rsincenu:tighness:test} passes, an interior point $\beta^0$ of a new cell has to be found (Line \ref{alg:rsincenu:new:beta}). This requires to perform ray-shooting: to find the first hyperplane (among $O(n^2)$ hyperplanes) that is intersected by a ray. Hence, the complexity is $O(n^2 p)$. Note, however, that the ray-shooting is performed only once per cell and thus this factor is dominated by the complexity $O(n\cdot\mathrm{lp}(n,p))$ of solving a LP for the cell.

    \emph{Part \ref{enu:the:enumeration:complexity}: Space complexity.}
    At most one linear program is needed at a moment. This gives the factor $\mathrm{lp}(n,p)$. Aside of this, the data are only copied and manipulated. There are $O(n^2)$ hyperplanes with $O(p)$ coefficients, which amounts to $O(n^2 p)$ in total.
 
\end{proof}

\begin{remark}
    Time complexity of FlIncEnu \cite{rada:2018:NewAlgorithmEnumeration} is $O(n^2 \cdot \abs{\xC}\cdot \mathrm{lp}(n^2,p))$. The benefit of RSIncEnu is in the lower number of linear programs to be solved and also in the fact that the linear programs are smaller.

    Actually, linear programs in RSIncEnu can be made even smaller using the idea of the variant of ReverseSearch algorithm \cite{sleumer:1998:Outputsensitivecellenumeration} and $\ell$-IE algorithm \cite{rada:2018:NewAlgorithmEnumeration}. However, the modification 
would not lead to an improvement of the asymptotic complexity bound.
\end{remark}

\section{The continuous and convex case ({$\textsf{CCC})$}}
\label{sect:CCC}

\subsection{Assumption implying continuity of the $F$-function}

\begin{assumption}\label{ass:cont}
Let $\alpha_1, \alpha_2, \dots, \alpha_n \in \mathbb{R}$ be given.
Let $\beta \in \bigcup_{\pi\in\xC} \xint(C^\pi)$ and let $\pi$ be the permutation such that
$
r^\beta_{\pi(1)} 
<
r^\beta_{\pi(2)} 
< \cdots
<
r^\beta_{\pi(n)}.
$
We assume that 
$$
a^{\beta}_i = \alpha_{\pi^{-1}(i)} \text{\ \ \ for all\ \ \ } i = 1, \dots, n.
$$
\end{assumption}
Now we can write
$$
F(\beta) = \sum_{i=1}^n \alpha_{\pi^{-1}(i)} r_i^{\beta} 
= \sum_{i=1}^n \alpha_{i} r_{\pi(i)}^{\beta}.
$$

Assumption \ref{ass:cont} implies that whenever 
we change $\beta$ to $\beta'$ in a way that 
their corresponding permutations $\pi, \pi'$ differ only in a transposition interchanging $\pi(k)$ and $\pi(\ell)$. Then coefficients $a_k$ and $a_{\ell}$ are exchanged. The remaining coefficients are the same.
In other words, when we change $\beta$ to $\beta'$ in a way that
a pair of residuals exchange their ranks, 
their $a$-coefficients are switched as well. 
Geometrically, we cross from a cell $C^{\pi}$ to a neighbor
cell $C^{\pi'}$. Now consider the line segment
$\{\gamma\beta + (1-\gamma)\beta':\ 0 \leq \gamma \leq 1\}$.
For some $\gamma_0 \in (0,1)$
we have $\beta_0 \coloneqq\gamma_0\beta + (1-\gamma_0)\beta'$ such that
$r_{\pi(k)}^{\beta_0} 
= r_{\pi(\ell)}^{\beta_0}$.
Since the residuals are equal, we also have
$$
\alpha_{k} r_{\pi(k)}^{\beta_0} 
+ \alpha_{\ell} r_{\pi(\ell)}^{\beta_0} 
=
\alpha_{\ell} r_{\pi(k)}^{\beta_0} 
+ \alpha_{k} r_{\pi(\ell)}^{\beta_0}.
$$ 
The other terms in $\sum_{i=1}^n \alpha_{i}r_{\pi(i)}^{\beta_0} = 
F(\beta_0)$ are unchanged.
It follows that
$\lim_{\gamma \nearrow \gamma_0} 
F(\gamma\beta + (1-\gamma)\beta')
= 
\lim_{\gamma \searrow \gamma_0} 
F(\gamma\beta + (1-\gamma)\beta')$. This (informal) argument shows, together with Assumption~\ref{as:reg}, that:

\begin{lemma} Under Assumptions \ref{ass:rank}, \ref{as:reg} and 
\ref{ass:cont}, the function $F$ is continuous.
\end{lemma}

\subsection{Monotonicity of rank coefficients and convexity of the $F$-function}

\begin{assumption}\label{ass:convex}
$\alpha_1 \leq \alpha_2 \leq \cdots \leq \alpha_n$.
\end{assumption}

Recall that in Section~\ref{sect:statistika} we discussed the typical choice of $\alpha_1, \dots, \alpha_n$ based on score functions as 
it is usual in statistical literature. This choice satisfies Assumption~\ref{ass:convex}.

\begin{lemma} Under Assumptions \ref{ass:rank}, \ref{as:reg},
\ref{ass:cont} and \ref{ass:convex}, the function $F$ is convex.
\label{lem:mx}
\end{lemma}

\begin{proof}
 We will prove the following Claim.
\emph{Let $\beta\in\mathbb{R}^p$ and let $\pi_0$ be any permutation satisfying
$r^\beta_{\pi_0(1)} 
\leq
r^\beta_{\pi_0(2)} 
\leq \cdots
\leq
r^\beta_{\pi_0(n)}.$ Let $\pi \in \mathcal{S}_n$. 
Then} 
\begin{equation}
F(\beta) = 
\sum_{i=1}^n \alpha_{i}r_{\pi_0(i)}^\beta
\geq
\sum_{i=1}^n \alpha_{i}r_{\pi(i)}^\beta.
\label{eq:qwe}
\end{equation}

The claim implies that $F(\beta)$ can be written in the form
\begin{equation}
F(\beta) = 
\max_{\pi\in\mathcal{S}_n} \sum_{i=1}^n 
\alpha_{i}r_{\pi(i)}^{\beta}
=
\max_{\pi\in\mathcal{S}_n} \sum_{i=1}^n 
\alpha_{i}(y_{\pi(i)} - x\T_{\pi(i)}\beta). \label{eq:maxform}
\end{equation}
It means that $F(\beta)$ is a finite maximum of linear functions, which is a convex function.

\emph{Proof of Claim.} Let
$k < \ell$, 
let a permutation $\pi$ satisfy $r_{\pi(k)}^\beta > r_{\pi(\ell)}^\beta$ and let $\pi'$ differ from $\pi$ 
only in a transposition interchanging $\pi(k)$ and $\pi(\ell)$. We show
that 
\begin{equation}
\sum_{i=1}^n \alpha_{i}r_{\pi'(i)}^{\beta} \geq \sum_{i=1}^n \alpha_{i}r_{\pi(i)}^{\beta}.
\label{eq:roste}
\end{equation}
Since the permutation $\pi_0$ can be obtained from any permutation $\pi$ by a finite number of transpositions, the claim follows.

To prove (\ref{eq:roste}), observe that $\alpha_\ell - \alpha_k \geq 0$ and
$
r_{\pi'(\ell)}^\beta
=
r_{\pi(k)}^\beta 
> 
r_{\pi(\ell)}^\beta
=
r_{\pi'(k)}^\beta.
$
Now
\begin{align*}
\sum_{i=1}^n \alpha_i r_{\pi'(i)}^\beta
&=
\sum_{i \neq k,\ell} \alpha_i r_{\pi'(i)}^\beta
+ \alpha_{k} r_{\pi'(k)}^\beta
+ \alpha_{\ell} r_{\pi'(\ell)}^\beta
=
\sum_{i \neq k,\ell} \alpha_i r_{\pi(i)}^\beta
+ \alpha_{k} r_{\pi'(k)}^\beta
+ \alpha_{\ell} r_{\pi'(\ell)}^\beta
\\ 
&=
\sum_{i \neq k,\ell} \alpha_i r_{\pi(i)}^\beta
+ \alpha_{k} r_{\pi(\ell)}^\beta
+ \alpha_{\ell} r_{\pi(k)}^\beta
\\&
\geq
\sum_{i \neq k,\ell} \alpha_i r_{\pi(i)}^\beta
+ \alpha_{k} r_{\pi(\ell)}^\beta
+ \alpha_{\ell} r_{\pi(k)}^\beta
- \underbrace{(\alpha_\ell - \alpha_k)}_{\geq 0}\underbrace{(r_{\pi(k)}^\beta - r_{\pi(\ell)}^\beta)}_{> 0}
\\&
=
\sum_{i \neq k,\ell} \alpha_i r_{\pi(i)}^\beta
+ \alpha_{k} r_{\pi(k)}^\beta
+ \alpha_{\ell} r_{\pi(\ell)}^\beta
= 
\sum_{i=1}^n \alpha_i r_{\pi(i)}^\beta.
\end{align*}
\end{proof}

To summarize: $F$ is continuous, convex and cell-wise linear on
the system of cells of the arrangement~\arr of hyperplanes (\ref{eq:hyperplanes}).

We conclude this section by a formal definition of the class 
\textsf{CCC}.

\begin{definition} \textsf{CCC} is the class of $F$-functions from
\eqref{eq:F} satisfying Assumptions~\ref{ass:rank}, 
\ref{as:reg}, \ref{ass:cont} and \ref{ass:convex}.\label{def:CCC}
\end{definition}

It is easy to see that the rank estimators known from robust statistics, as described in
Section~\ref{sect:statistika}, belong to \textsf{CCC}.

\subsection{Representation of faces of an arrangement and an example}
Recall that $F$ is a cell-wise linear function on $\xC$, the system of cells $C^{\pi}$ of the arrangement~\arr.
Recall also that non-empty intersections of cells are called \emph{faces} (of the arrangement).

A vertex of the arrangement can be described as an intersection of $p$ independent hyperplanes from \eqref{eq:hyperplanes}, i.e.~as a particular linear system with a unique solution. 
Analogously, if $\Phi$ is a face of dimension $d$, the affine hull of $\Phi$ can be described as an intersection of $p-d$ independent hyperplanes from \eqref{eq:hyperplanes}.

When we are given a family of faces, its \emph{minimal face} is (any) face from the family with the lowest dimension.

\begin{example}
 Figure~2 shows an example of $F(\beta) \in\textsf{CCC}$
with $\beta \in \mathbb{R}^2$, $n = 5$ and 
$\alpha = (-2, -1, 0, 1, 2)\T$. Here, a minimizer exists but is not unique: all points of the cell $C^{\pi}$ are minimizers. However, there are only four \emph{minimal faces} where the minimum is attained --- these are the four vertices of $C^{\pi}$.

\begin{figure}[t]
\centering
\includegraphics[width=10cm]{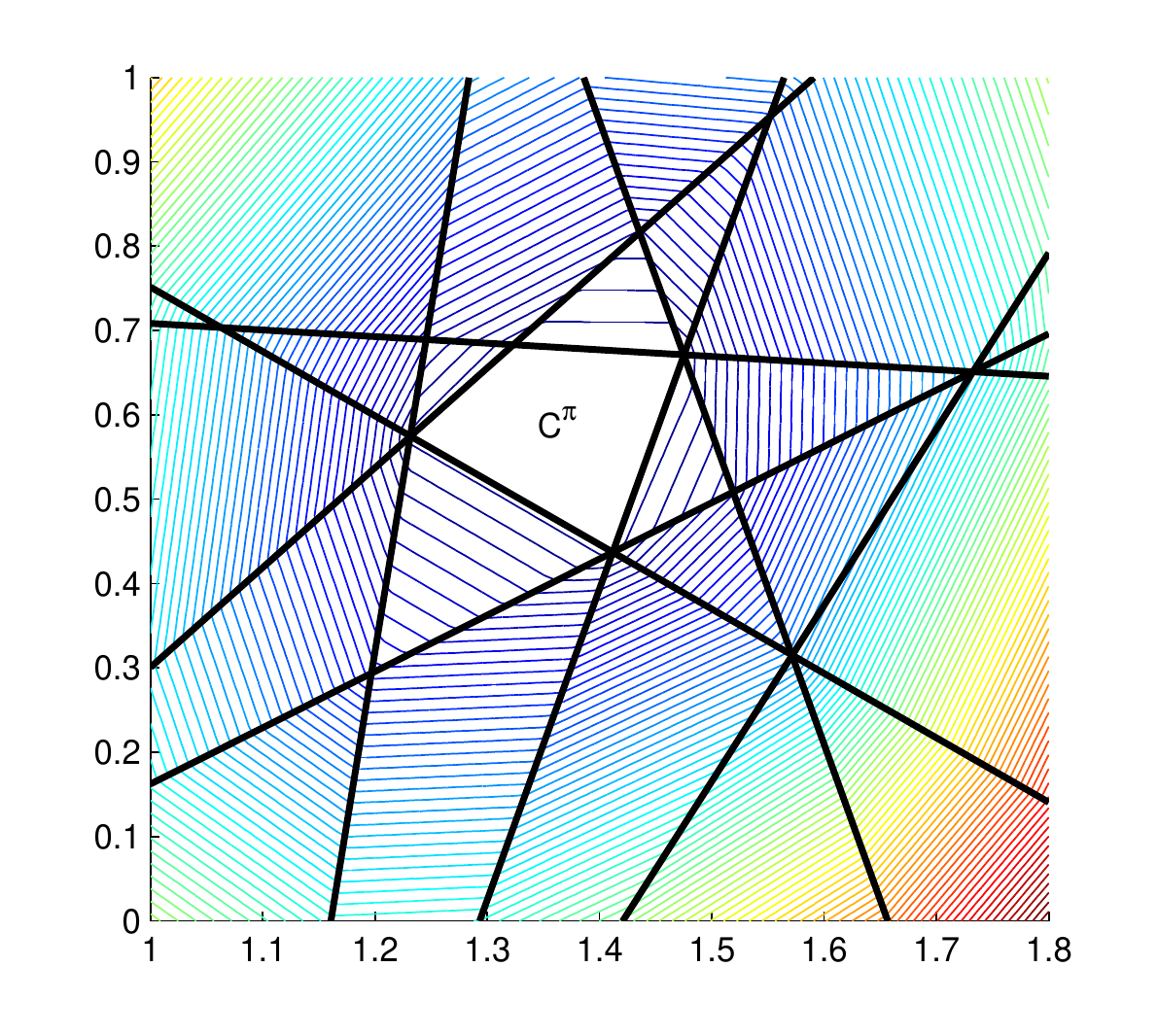}
\caption{An example of the arrangement \arr{} in $\mathbb{R}^2$ with $n = 5$
and contour lines of a sample function in \textsf{CCC}. The minimum of $F(\beta)$ is attained on the cell $C^{\pi}$.}
\end{figure}
\end{example}

\subsection{Bit sizes of rational numbers and some Big-$L$ arguments}\label{sect:BigL}

If $\gamma = \pm \frac{\varrho}{\vartheta}$ is a rational number written down in the coprime form, then $\bitsize(\gamma)$ is the number of bits
needed to write down the numerator $\varrho$ and denominator $\vartheta$ in binary (plus a bit to write down the sign). Asymptotically, $\bitsize(\gamma) = O(\log \varrho + \log \vartheta)$. When $\Gamma \coloneqq (\gamma_1, \dots, \gamma_N)$ is a family of rational numbers, then $\bitsize(\Gamma) = \sum_{i=1}^N\bitsize(\gamma_i)$. We will use this notation quite freely, for example $\bitsize(X,y)$ is the sum of bitsizes of the entries in the (rational) matrix $X$ and the (rational) vector $y$.

Recall that our task is to find a minimizer of a function $F \in \textsf{CCC}$.
The function $F$ is fully described by $(X,y,\alpha)$, where 
$\alpha = (\alpha_1, \dots, \alpha_n)$ are the coefficients from Assumption~\ref{ass:cont}. Thus the size of the input instance is
$$
L \coloneqq \bitsize(X,y,\alpha).
$$
The task is to describe an algorithm the working time of which can be bounded by a polynomial in $L$. 

We will also need some easy technical 
properties which are easily proved along the lines of 
Part~I of \cite{schrijver:2000:TheoryLinearInteger}.

\begin{lemma} \label{lem:q}
There is a polynomial $q$ with the following properties:
\begin{itemize}
\item[(a)]
Each face $\Phi$ of the arrangement  
\arr{} 
has a point $b \in \Phi$ with $\bitsize(b) \leq q(L)$.
In particular, every vertex $v$ of the arrangement has
$\bitsize(v) \leq q(L)$.
\item[(b)] If there exists 
\begin{equation}
t_0 \coloneqq \min_{\beta\in\mathbb{R}^p} F(\beta),
\label{eq:tzero}
\end{equation}
then $\bitsize(t_0) \leq q(L)$.
Moreover, if there exists a minimizer, there also exists a minimizer $\beta_0$ with
$\bitsize(\beta_0) \leq q(L)$.
\end{itemize}
\end{lemma}

\begin{proof}
    Let $\Phi$ be a face and $\Phi'$ be a minimal face such that $\Phi' \subseteq \Phi$.\footnote{For the sake of simplicity, a reader can always think of $\Phi'$ as a vertex; but generally, an arrangement need not have vertices, so we prefer this more careful formulation.} 
Since $\Phi'$ is minimal, 
$\Phi' = \text{affine.hull}(\Phi')$. Thus $\Phi'$ can be described
by a system of $d$ independent equations from 
\eqref{eq:hyperplanes}, where $d \leq p$ is the dimension of $\Phi'$. 
The bitsize of the system can be obviously bounded by $L$ and 
a solution $b$ of the system can be found in polynomial time. 
Thus, $\bitsize(b) \leq q'(L)$ for some polynomial $q'$.

To prove (b), let $\Phi$ be a minimal face in which the argmin is attained. By (a), this face contains a point $\beta_0$ with $\bitsize(\beta_0) \leq q'(L)$. Since $F$ is a function computable in polynomial time, the bitsize of its output $t_0 = F(\beta_0)$ is polynomially bounded by the bitsize of the input $\beta_0$. It follows that the bitsize of $t_0$ is also bounded 
by a polynomial in $L$.
\end{proof}

\subsection{The polynomial algorithm for functions in $\textsf{CCC}$}

First we will need a YES-NO oracle for the decision problem 
``given $t$, does there exist $\beta$ such that $F(\beta) \leq t$?''.
This oracle is based on the ellipsoid method and will be described in Section~\ref{elli}. Here we use it as a blackbox. The necessary properties are summarized by the following lemma. 

We formulate the oracle in a slightly more general way, which will be useful below.

\begin{lemma}\label{lem:oracle}
There exists a YES-NO algorithm $\mathfrak{A}$ for the following decision problem: 
``given $(X \in \mathbb{Q}^{n\times p}$, 
        $y \in \mathbb{Q}^n$, 
				$\alpha \in \mathbb{Q}^n$, 
				$t \in \mathbb{Q}$,
				$W \in \mathbb{Q}^{k \times p}$,
				$z \in \mathbb{Q}^{p})$,			
does there exist a $\beta\in \mathbb{R}^p$ such that
\begin{itemize}
\item[(a)] $W\beta \leq z$ and 
\item[(b)] $F(\beta) \leq t$?''
\end{itemize}
Moreover, if the bitsizes of $W, z, t$ are polynomially bounded in $L$,
then algorithm $\mathfrak{A}$ works 
in polynomial time with respect to $L$.
\end{lemma}

Sometimes, property (a) will not be necessary; then we will use the notation
$$
\mathfrak{A}^*(X,y,\alpha,t) \coloneqq
\mathfrak{A}(X,y,\alpha,t, W\coloneqq (0, \dots, 0), z \coloneqq 0).
$$

We will also need the following form of Diophantine approximation.
Recall that Diophantine approximation is a method for finding a rational number with a small denominator close to a given rational number.

\begin{lemma}[Diophanine approximation, Corollary 6.3a in 
\cite{schrijver:2000:TheoryLinearInteger}]\label{lem:dioph} \mbox{\ }
\begin{itemize}
\item[(a)] Given $\gamma \in \mathbb{Q}$ and $M \in \mathbb{N}$,
there exists at most one rational number $\frac{\varrho}{\vartheta}$
satisfying 
\begin{equation}
\text{(i)\ }1 \leq \vartheta \leq M\text{\qquad and\qquad (ii)\ }\left|\gamma - \frac{\varrho}{\vartheta}\right| < \frac{1}{2M^2}.
\label{eq:daprx}
\end{equation}
(Here, $\varrho \in \mathbb{Z}$ and $\vartheta \in \mathbb{N}$.)
\item[(b)] There is a polynomial-time algorithm 
$\mathcal{D}(\gamma, M)$
which tests if the number $\frac{\varrho}{\vartheta}$ exists, and if so, the algorithm finds it.
\end{itemize}
\end{lemma}

Now we are ready to formulate the body of the $\textsl{CCC}$-algorithm. Recall that we are given $(X,y,\alpha)$ as input and the task is to find (some) $\beta_0$ such that 
$$F(\beta_0) = t_0, \text{\ \ where\ \ } 
t_0 = \min_{\beta\in\mathbb{R}^p} F(\beta),
$$ 
or assert that the problem is unbounded (from below). 

\myparagraph{Notation} When $A,B$ are matrices with the same number of columns, 
then $\binom{A}{B}$ stands for their concatenation, i.e.~the matrix resulting from sticking $B$ under $A$.

\begin{algorithm}[!h]
    \begin{algorithmic}[1]
        \Require {$X \in \QQ^{n\times p},y \in \QQ^n,\alpha \in \QQ^n$ s.t. $\alpha_1 \leq \cdots \leq \alpha_n$}
        \State $L \coloneqq \bitsize(X,y,\alpha)$;\quad let $q$ be the polynomial from Lemma~\ref{lem:q}
        \vskip5pt
        \Statex \textit{Unboundedness test:}
        \State \textbf{If} $\mathfrak{A}^*(X,y,\alpha,t=-2^{q(L)}-1) = \text{``YES''}$ \textbf{then}
        report ``unbounded'' and stop\label{alg:ccc:unboudedness}
        \vskip5pt
        \Statex \textit{Binary search to localize $t_0$ in a narrow interval $[\underline{t}, \overline{t}]$:}
        \State $\overline{t} \coloneqq 2^{q(L)}$;\quad $\underline{t} \coloneqq -2^{q(L)}$\label{alg:ccc:initial:bounds}
        \While {$\overline{t} - \underline{t} > 2^{-2q(L)-1}$} \label{alg:ccc:while:start}
        \State $t^* \coloneqq \frac{1}{2}(\overline{t} + \underline{t})$ \label{alg:ccc:update:step}
        \State \textbf{If} $\mathfrak{A}^*(X,y,\alpha,t^*) = \text{``YES''}$ \textbf{then} $\overline{t} \coloneqq t^*$ \textbf{else} $\underline{t} \coloneqq t^*$\label{alg:ccc:while:inside:test}
        \EndWhile \label{alg:ccc:while:end}
        \vskip5pt
        \Statex \textit{Retrieve $t_0$ from $[\underline{t}, \overline{t}]$ by Diophantine approximation:}
        \State$t_0 \coloneqq \mathcal{D}(\gamma = \frac{1}{2}(\overline{t} + \underline{t}), M = 2^{q(L)})$\label{alg:ccc:diophantine}
        \vskip5pt
        \Statex\textit{Find a minimizer $\beta_0$:}
        \State Let $W$ be an empty matrix and $z$ an empty vector\label{alg:ccc:for:initialization}
        \For{$1\le i < j \le n$}\label{alg:ccc:for:start}
        \State $W^* \coloneqq \binom{W}{x_i\T-x_j\T}$;\quad $z^* \coloneqq \binom{z}{y_i-y_j}$\label{alg:ccc:for:matrix:updates}
        \State \textbf{If} $\mathfrak{A}\left(X,y,\alpha,t_0,\binom{W^*}{-W^*},\binom{z^*}{-z^*}\right) =  \text{``YES''}$ \textbf{then} $W\coloneqq W^*$, $z \coloneqq z^*$\label{alg:ccc:for:oracle:calls}
\EndFor\label{alg:ccc:for:end}
        \vskip5pt
        \Ensure{}
        \State A minimizer $\beta_0$ is (any) solution of the system $W\beta_0 = z$ \label{alg:ccc:output}
    \end{algorithmic}
    \caption{\textsl{CCC}-algorithm}
    \label{alg:ccc}
\end{algorithm}

\myparagraph{The algorithm}
The CCC-algorithm is formalized as Algorithm \ref{alg:ccc}.
\begin{theorem}
The \textsl{CCC}-algorithm is correct and works in polynomial time.
\end{theorem}

\begin{proof}
By Lemma \ref{lem:q}(b), $\bitsize(t_0) \leq q(L)$. 
A number, which can be written down by $q(L)$ bits, can be at most
$2^{q(L)}$ in absolute value. Thus, if the problem is bounded, then
$t_0 \geq -2^{q(L)}$. It follows that if there exists $\beta$ such that $L(\beta) \leq -2^{q(L)} - 1$, then the problem is unbounded. This proves correctness of Line~2.

The same argument proves that for the initial bounds set in Line \ref{alg:ccc:initial:bounds}, we have
\begin{equation}\label{eq:t}
\underline{t} \leq t_0 \leq \overline{t}, 
\end{equation}
and by the update step on Line~6 it follows that \eqref{eq:t} holds true
also after the last iteration of the while-cycle on Lines \ref{alg:ccc:while:start}--\ref{alg:ccc:while:end}.

\emph{Proof of correctness of Line \ref{alg:ccc:diophantine}.} In Line~\ref{alg:ccc:diophantine} we have 
$$
\overline{t} - \underline{t} \leq 2^{-2q(L)-1} 
$$
and \eqref{eq:t} holds true.
We run the Diophantine approximation of 
$$
\gamma = \tfrac{1}{2}(\overline{t}+\underline{t}) \text{\ \ \ with\ \ \ } M = 2^{q(L)}.
$$
\emph{Observation. The number $t_0 =: \frac{\varrho}{\vartheta}$ satisfies (i) and (ii) in \eqref{eq:daprx}}.
Proof.
Since 
$\bitsize(\vartheta) \leq \bitsize(t_0) \leq q(L)$, we have
$1 \leq \vartheta \leq 2^{q(L)} = M$. Thus (i) is satisfied.
Now 
$$
\left|\gamma - \frac{\varrho}{\vartheta}\right| = |\gamma - t_0| \leq \frac{1}{2}(\overline{t}-\underline{t}) \leq \frac{1}{2} \cdot 
2^{-2q(L)-1} = \frac{1}{2} \cdot \frac{1}{2M^2} < \frac{1}{2M^2}.
$$
So, (ii) holds true and the proof of the observation is complete.

It follows that $t_0 = \frac{\varrho}{\vartheta}$ is exactly the number which is found
by the algorithm from Lemma~\ref{lem:dioph}(b). (And the algorithm cannot terminate with a report that no such number exists.) The proof of correctness of Line \ref{alg:ccc:diophantine} is complete.

The for-cycle on Lines \ref{alg:ccc:for:start}--\ref{alg:ccc:for:end} finds the affine hull of a minimal face of the arrangement \arr{} which contains a minimizer of $F(\beta)$. 
First we start the search in the entire space 
$\mathbb{R}^p$ (Line \ref{alg:ccc:for:initialization}). Then, on Lines \ref{alg:ccc:for:matrix:updates}--\ref{alg:ccc:for:oracle:calls}, 
we are asking the oracle $\mathfrak{A}$ whether or not 
a restriction of the search space into the hyperplane
$H_{ij} = \{\beta:\ (x_i - x_j)\T\beta = y_i - y_j\}$ cuts off all minimizers. If it does, we drop the restriction; otherwise we retain it (see the ``then'' part of Line \ref{alg:ccc:for:oracle:calls}). 

Finally, in Line \ref{alg:ccc:output}, we get a linear system 
\begin{equation}
W\beta_0 = z 
\label{eq:finsyst}
\end{equation}
which describes the affine hull of the minimal face $\Phi$ containing $\beta_0$. Since $\Phi$ is minimal, $\Phi = \text{affine.hull}(\Phi)$ and 
it is correct to take \emph{any} solution $\beta_0$ 
of \eqref{eq:finsyst}.

\emph{Polynomiality.} It is easy to show that the algorithm works in
polynomial time. The number of iterations of the for-cycle on Lines \ref{alg:ccc:for:start}--\ref{alg:ccc:for:end} is $O(n^2)$. The number of iterations of the while-cycle
on Lines \ref{alg:ccc:while:start}--\ref{alg:ccc:while:end} is bounded by 
$$
O\left(\log_2 \frac{2 \cdot 2^{q(L)}}{2^{-2q(L)-1}}\right) = O(q(L)).
$$

Then, observe that whenever we call the oracle 
$\mathfrak{A}$, the size of input $(W, z, t)$ is polynomially bounded in $L$. (Namely, each of the numbers $\underline{t}, \overline{t}$ inside the iterations of while-cycle on Lines \ref{alg:ccc:while:start}--\ref{alg:ccc:while:end}, can be written down with denominator not exceeding $2^{2q(L)+1}$.) By Lemma~\ref{lem:oracle}, a call of $\mathfrak{A}$ consumes polynomial time and there is a polynomial number of such calls. And, by Lemma~\ref{lem:dioph}, Line \ref{alg:ccc:diophantine} works in polynomial time since the input $(\gamma, M)$ has polynomially bounded bitsize. 
\end{proof}

\subsection{Construction of the oracle $\mathfrak{A}$: 
Proof of Lemma~\ref{lem:oracle}}\label{elli}

The algorithm $\mathfrak{A}$ is based on a version of the ellipsoid method with central cuts, a membership oracle and a separation oracle. We assume that a reader is familiar with the mechanics and proof tricks around the ellipsoid method, which are mostly based on Big-$L$ arguments. 
Excellent sources are \cite{schrijver:2000:TheoryLinearInteger} and \cite{grotschel:1993:Geometricalgorithmscombinatorial}.
The only ``really interesting'' part of our algorithm is the construction of the membership oracle and the separation oracle. The rest is a standard in the ellispoid-method theory and is described only with (tacit) references to the above mentioned books. 

Recall that the algorithm $\mathfrak{A}$ gets $(X, y, \alpha, t, W, z)$ as input and decides whether there exists $\beta$ such that
\begin{subequations}
    \label{eq:conds}
\begin{align}
    \label{eq:cond:a}W\beta &\leq z,\\
    \label{eq:cond:b}F(\beta)& \leq t.
\end{align}
\end{subequations}

The core of $\mathfrak{A}$ is in reformulation of this question
into a linear programming form. Condition \eqref{eq:cond:a} is a system of linear inequalities. The following lemma shows how to reformulate \eqref{eq:cond:b}. 

\begin{lemma}
    Let $t \in \mathbb{R}$. Then $F(\beta) \leq t$ if and only if 
$\beta$ satisfies the system
\begin{equation}
t \geq \sum_{i=1}^n \alpha_{i}(y_{\pi(i)} - x_{\pi(i)}\T\beta) 
\quad \forall \pi \in \mathcal{S}_n\label{eq:permx}.
\end{equation}
\end{lemma}

\begin{proof}
Directly from \eqref{eq:maxform}.
\end{proof}

\myparagraph{The main obstacle}
System \eqref{eq:permx} has a superpolynomial number of inequalities.
This is why we cannot simply write system \eqref{eq:cond:a}--\eqref{eq:cond:b} as a single linear programming problem
\begin{equation}
W\beta \leq z, \quad
t \geq \sum_{i=1}^n \alpha_{i}(y_{\pi(i)} - x_{\pi(i)}\T\beta) 
\quad \forall \pi \in \mathcal{S}_n
\label{eq:mainsyst}
\end{equation}
(where $t$ is a fixed parameter and $\beta$ are variables).
Neither can we work with the dual problem
since it has an excessive number of variables.

\myparagraph{Details of the ellipsoid method}
To overcome the obstacle, the right tool is the \emph{ellipsoid method with separation oracle}
as described in~\cite{grotschel:1993:Geometricalgorithmscombinatorial}. 

We will use
the version with central cuts. Given an ellipsoid
\begin{displaymath}\mathcal{E}(E,c) \coloneqq \{b \in \mathbb{R}^p: (b - c)\T E^{-1} (b-c) \leq 1\}, \end{displaymath}
where $E$ is a positive definite matrix (and $c$ is the \emph{center}),
and a vector $s \neq 0$, the \emph{central cut} generates the half-ellipsoid $\mathcal{H}\coloneqq \mathcal{E}(E,c) \cap 
\{b: (b - c)\T s \leq 0\}$. Then, $\mathcal{H}$ is covered by the minimum-volume (L\"owner-John) ellipsoid $\mathcal{E}(E',c')$, which is known to have the form
\begin{equation}
E' = \frac{n^2}{n^2-1}\left(E - \frac{2}{n+1}\frac{Es s\T E}{s\T E s}
\right),
\quad
c' = c - \frac{1}{n+1}\frac{Es}{\sqrt{s\T E s}}.
\label{eq:cut}
\end{equation}

To specify (a version of) the ellipsoid algorithm in full, 
one must also specify the initial ellipsoid, the methods for testing feasibility (``membership oracle''), 
the method for generation of the separator $s$ (``separation oracle'') and the termination criterion. First we summarize some technical details. Observe that \eqref{eq:mainsyst} defines a convex polyhedron in $\mathbb{R}^p$; let us denote it by $\mathcal{P}$.

\begin{lemma}\label{lem:voser} 
There are polynomials $q', q'', q'''$ with the following properties.
\begin{itemize}
\item[(a)] Let 
\begin{multline*}
\mathcal{P}' \coloneqq 
\Big\{b:\ Wb \leq z + 2^{-q'(L)}e, \ \
t + 2^{-q'(L)} \geq \sum_{i=1}^n \alpha_{\pi(i)}(y_{i} - x_{i}\T b)\ \forall \pi \in \mathcal{S}_n\Big\},
\end{multline*}
where $e = (1, \dots, 1)\T$.
The polynomial $q'(L)$satisfies that 
$\mathcal{P}' = \emptyset$ if and only if $\mathcal{P} = \emptyset$. Moreover, if $\mathcal{P}' \neq \emptyset$, then 
$\mathcal{P}'$ is full-dimensional with
$\text{volume}(\mathcal{P}') \geq 2^{-q''(L)}$.
\item[(b)] $\mathcal{P'} = \emptyset$
if and only if 
$\mathcal{P'} \cap 
\{b: -2^{q'''(L)}e \leq b \leq 2^{q'''(L)}e\} = \emptyset$.
\end{itemize} 
\end{lemma}

\begin{proof} Both arguments (a), (b) are standard in the theory of the ellipsoid method, see e.g.~\cite{schrijver:2000:TheoryLinearInteger}, Chaper~13. The standard arguments are valid in our context, too, because the 
$\textsf{CCC}$-algorithm
uses oracle $\mathfrak{A}$ \emph{only with $t, W, z$ with bitsizes  
bounded by a polynomial in $L$} (this is a crucial point). So, the standard arguments allow us to derive the polynomials $q', q'', q'''$ from the polynomial $q$ in Lemma~\ref{lem:q}.
\end{proof}

The ellipsoid method is run with the blown-up polyhedron $\mathcal{P}'$
for which we apply Lemma \ref{lem:voser}(b). It implies that we can construct the initial
ellipsoid as 
\begin{equation}\label{eq:inieli}
\text{the smallest Euclidian ball covering the cube 
$\{b: -2^{q'''(L)}e \leq b \leq 2^{q'''(L)}e\}$.}
\end{equation} 
The estimate on volume from Lemma \ref{lem:voser}(a) implies that the method terminates after
a polynomial number of steps; this is the same argument as in 
\cite{schrijver:2000:TheoryLinearInteger}, Theorem~13.4. To recall shortly, just for the sake of completeness: the initial ellipsoid
is easy shown to have volume at most $2^{\widetilde q(L)}$, where
$\widetilde q$ is a polynomial, and one application
of the cut \eqref{eq:cut} decreases the volume by factor 
$\exp(-\frac{1}{2p+2}) < 1$. Since volume$(\mathcal{P}') \geq 2^{-q''(L)}$,
after at most $O(p + \widetilde{q}(L) + q''(L)) = O(L + \widetilde{q}(L) + q''(L))$ 
iterations the volume decreases under the lower bound $2^{-q''(L)}$. When this number of steps
is exhausted without finding a point in $\mathcal{P}'$, the result of
the algorithm $\mathfrak{A}$ is ``NO''.

\myparagraph{The work inside an iteration: the membership oracle and the separation oracle}
What remains to be described is the work inside an iteration. Inside an iteration, we have an ellipsoid centered in $c$ and we must first decide whether $c \in \mathcal{P}'$. If the answer is negative, we must construct the separator.

\emph{Membership oracle.} Given the center $c$, the condition
\begin{equation}\label{eq:Cone}
Wc \leq z + 2^{-q'(L)}
\end{equation} 
can be tested directly.
Then, evaluate $F(c)$. Now we have 
\begin{equation}\label{eq:Ctwo}
F(c) \leq t + 2^{-q'(L)}
\end{equation} if and only if $c \in \mathcal{P}'$.
To summarize: if both conditions \eqref{eq:Cone}, \eqref{eq:Ctwo} are satisfied, the ellipsoid method has found
out that $c \in \mathcal{P}'$, implying that $\mathcal{P} \neq \emptyset$ and that the answer of the algorithm 
$\mathfrak{A}$ is ``YES''.

\emph{Separation oracle.} If there is an $i$ such that
$
w_i\T c > z_i + 2^{-q'(L)}, 
$
where $w_i\T$ is the $i$th row of $W$ (where we can, without loss of generality, assume $w_i \neq 0$),
then the separation oracle outputs $s = w_i$.

Otherwise, we find (any) permutation $\pi$ such that
\begin{equation}
r^c_{\pi(1)} \leq r^c_{\pi(2)} \leq \cdots \leq r^c_{\pi(n)};
\label{eq:permcc}
\end{equation} 
this can be done in polynomial time by sorting. Then the oracle outputs 
\begin{equation}
s = -\sum_{i=1}^n \alpha_{i}x_{\pi(i)}.
\label{eq:grad}
\end{equation}

The fact that both the membership oracle and the separation oracle work in polynomial time is obvious. 
The correctness of the separation procedure follows from the next lemma.

\begin{lemma} There exists no $b$ such that $(b-c)\T s > 0$ and 
$F(b) \leq t + 2^{-q'(L)}$.
\end{lemma}

\begin{proof}
For contradiction assume that there exists $b$ such that 
\begin{equation}
b\T s > c\T s\ \ \ \text{and}\ \ \ F(b) \leq t + 2^{-q'(L)}. 
\label{eq:assq}
\end{equation}
Let $\pi$ be a permutation from 
\eqref{eq:permcc}.
Define $\eta \coloneqq \sum_{i=1}^n\alpha_{i} y_{\pi(i)}$
and 
observe that
\begin{align*}
c\T s +\eta & = s\T c +\eta =
-\sum_{i=1}^n \alpha_{i}x\T_{\pi(i)}c + 
\sum_{i=1}^n\alpha_{i} y_{\pi(i)}
\\ &= 
\sum_{i=1}^n \alpha_{i}(y_{\pi(i)} - x_{\pi(i)}\T c)
= F(c).
\end{align*}
Similarly, if $\pi'$ is a permutation such that
$
r^b_{\pi'(1)} \leq r^b_{\pi'(2)} \leq \cdots \leq r^b_{\pi'(n)},
$
then 
\begin{align*}
b\T s +\eta & = s\T b +\eta =
-\sum_{i=1}^n \alpha_{i}x\T_{\pi(i)}b +
\sum_{i=1}^n\alpha_{i} y_{\pi(i)}
\\ &= 
\sum_{i=1}^n \alpha_{i}(y_{\pi(i)} - x_{\pi(i)}\T b)
\leq 
\sum_{i=1}^n \alpha_{i}(y_{\pi'(i)} - x_{\pi'(i)}\T b)
= F(b).
\end{align*}
The inequality follows from \eqref{eq:qwe}.

Now we have a contradiction:
$$
t + 2^{-q'(L)} 
\stackrel{\text{(a)}}{\geq} 
F(b) 
\geq 
b\T s + \eta 
\stackrel{\text{(b)}}{>}  c\T s + \eta
= F(c) 
\stackrel{\text{(c)}}{>}  t + 2^{-q'(L)}.
$$ 
Inequalities (a) and (b) follow from 
\eqref{eq:assq} and inequality (c) follows from
the fact that the membership oracle did not succeed in test \eqref{eq:Ctwo}. 
\end{proof}

This concludes the description of algorithm $\mathfrak{A}$ and the proof
of Lemma~\ref{lem:oracle}. 

\subsection{An illustration}

\begin{figure}[tb]
\centering
\includegraphics[width=14cm]{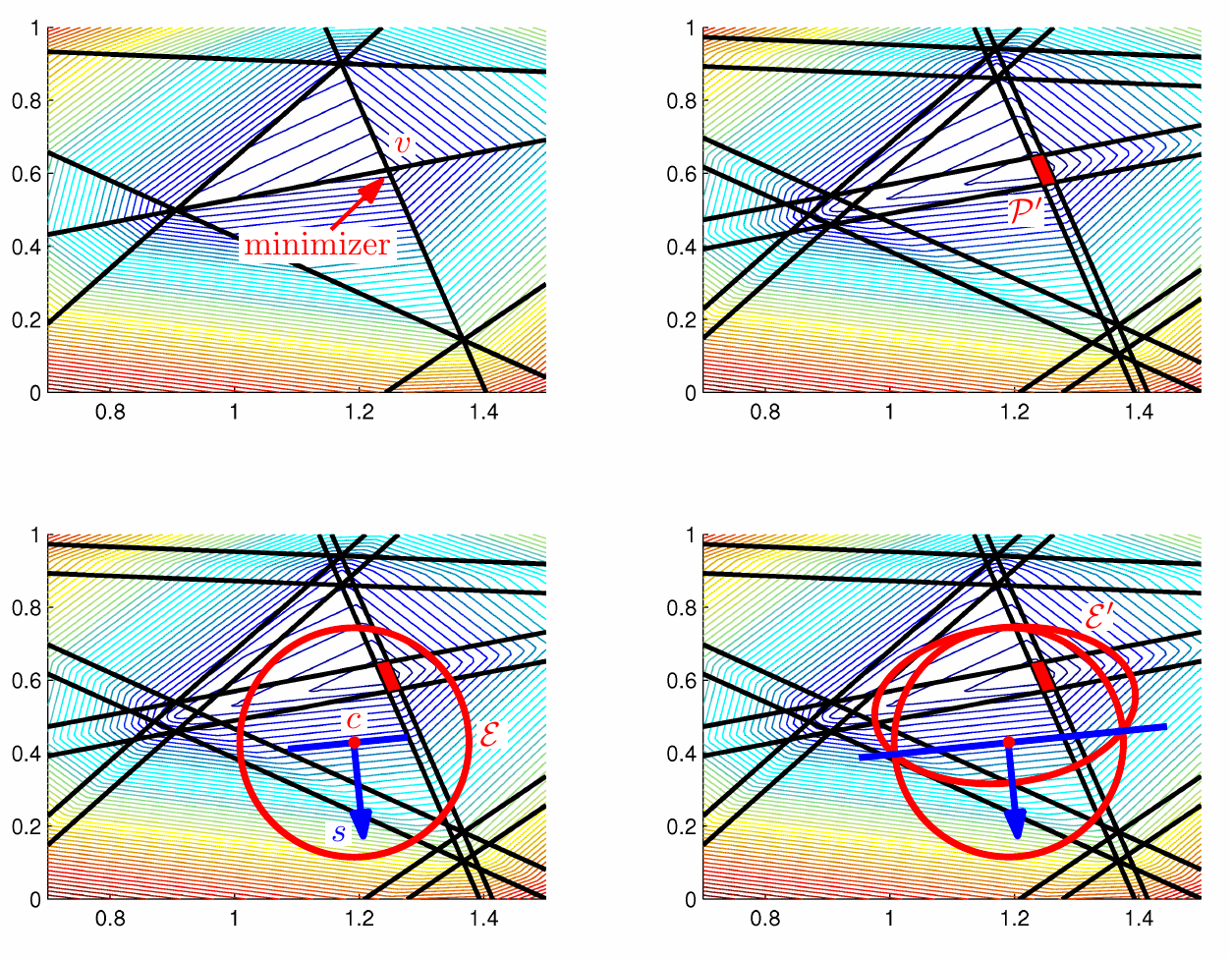}
\caption{An illustration how the ellipsoid method works on the arrangement \arr.}
\end{figure}

Figure~3 depicts the ellipsoid method and its relation to
the arrangement \arr. The first picture shows an example of a function $F(\beta) \in \textsf{CCC}$ where the minimizer is unique. The minimizer is in a vertex $v$ of the arrangement~\arr. Assume that algorithm 
$\mathfrak{A}$ is called with $t = \min F(\beta)$ and assume that the system $W\beta \leq z$ is empty. Then, polyhedron $\mathcal{P}$ given by
\eqref{eq:mainsyst} is a singleton $\{v\}$. It is dimension-deficient and cannot be localized by the ellipsoid method. 

The second picture illustrates the blow-up step (i.e.,~the replacement of $\mathcal{P}$ by $\mathcal{P}'$)
from Lemma~\ref{lem:voser}(a). The blow-up step can be understood as a replacement of a hyperplane from the arrangement 
by a narrow but full-dimensional band; see also step \{12\} of the CCC-algorithm. Now, the singleton $\mathcal{P} = \{v\}$ has been ``expanded'' into the full-dimensional region $\mathcal{P}'$. Moreover, by Lemma \ref{lem:voser}(a), the volume of $\mathcal{P}'$ is at least $2^{-q''(L)}$.
The region $\mathcal{P}'$ has the property
that $F(\beta) \leq 
t + 2^{-q'(L)}$ for all $\beta \in \mathcal{P}'$. A~point in $\mathcal{P}'$ is to be found.

The third picture shows an example of an ellipsoid $\mathcal{E}$ centered at $c$ which can occur among the iterations of the ellipsoid method. Clearly $c \not\in \mathcal{F}$, and thus the membership oracle replies ``NO''. The separation oracle
returns $s$ given by \eqref{eq:grad}; in fact, it is the gradient of $F(\beta)$ in $c$ (geometrically: it is a vector orthogonal to the contour line of $F$ in $c$). 

\emph{Remark.} In general, the gradient in $F(c)$ need not exist. However, the choice of $\pi$ in \eqref{eq:permcc} and 
\eqref{eq:grad} is still correct. (The selection of $\pi$ in 
\eqref{eq:permcc} can be seen as a selection 
of a representative from the polar cone of the graph of $F$ in $c$, sometimes referred to as \emph{sub-gradient}).

The fourth picture shows the next ellipsoid $\mathcal{E}'$ constructed
by \eqref{eq:cut}. Observe that initially we had $\mathcal{P}' \subseteq \mathcal{E}$ and now we have $\mathcal{P}' \subseteq \mathcal{E}'$. This invariant can be preserved 
from the very first iteration.
Indeed, for every vertex $w$ of $\mathcal{P}'$ it holds true that
$\bitsize(w)$ is polynomially bounded in $L$, and
all vertices (and thus the entire region $\mathcal{P}'$)
can be covered by the initial ellipsoid \eqref{eq:inieli} by a suitable choice of the polynomial $q'''$. (This insight can be used only if $\mathcal{P}$ is bounded as in the example from Figure~3.)


\section{Concluding remarks and challenges} 

We defined two classes of functions, \textsf{GEN} and \textsf{CCC}, sharing a common property: they are cell-wise linear functions
on a certain arrangement of hyperplanes. Our motivation came from robust regression, where rank estimators are of central importance;
the estimators form a subclass of \textsf{CCC}-functions. The class \textsf{GEN} is a natural generalization.

Optimization of a \textsf{GEN}-function
is (a kind of) combinatorial optimization problem and requires an exhaustive search 
over the entire system of cells. (If an algorithm omits a cell $C$, it is always possible to place the minimizer to $C$.) We have utilized special properties of the arrangement to design an algorithm with better time complexity
than previously known methods for general arrangements. This is the main result of Section~\ref{sect:general}.
The algorithm works in polynomial time as long as the dimension is assumed to be fixed.

The main result of Section~\ref{sect:CCC} is that \textsf{CCC}-functions can be minimized in polynomial time. This result is based on a linear programming formulation with $n!$ constraints. This huge number of constraints prevents us from using standard LP techniques (it is even impossible to work with the dual problem due to the excessive number of variables). 
This obstacle is overcome by the ellipsoid method with membership and separation oracles.

Strictly speaking, the CCC-algorithm can be expected to suffer from many numerical problems, e.g.~because of the rounding step based on Diophantine approximation. But, from another viewpoint, CCC-algorithm can be interpreted positively: the separator
(\ref{eq:grad}) can be regarded as a form of ``gradient'' descent method with a careful choice of the step length guaranteeing polynomial convergence. (We put ``gradient'' into quotation marks since the function is non-smooth.)

The main challenge for now is obvious: to design an interior point method for \textsf{CCC}-functions. 

\section*{Acknowledgement}
The work was supported by the Czech Science Foundation (the first, the third and the fourth author: 19-02773S, the second author: 17-13086S).

\bibliographystyle{tfnlm}
\bibliography{rank_statistiky} 

\end{document}